%% file: RC.tex
\documentclass[12pt,reqno]{amsart}
\usepackage{amssymb,amscd}
\usepackage[alphabetic]{amsrefs}
\usepackage[all]{xy}
\usepackage[headings]{fullpage}
\usepackage[T1]{fontenc}
\usepackage{libertine}

\include{formatting}
\include{macros}

\begin{document}
\title{Rational curves on elliptic surfaces}

\author{Douglas Ulmer}
\address{School of Mathematics \\ Georgia Institute of Technology
  \\ Atlanta, GA 30332}
\email{ulmer@math.gatech.edu}

\date{\today}


\begin{abstract}
  We prove that a very general elliptic surface $\EE\to\P^1$ over the
  complex numbers with a section and with geometric genus $p_g\ge2$
  contains no rational curves other than the section and components of
  singular fibers.  Equivalently, if $E/\C(t)$ is a very general
  elliptic curve of height $d\ge3$ and if $L$ is a finite extension of
  $\C(t)$ with $L\cong\C(u)$, then the Mordell-Weil group $E(L)=0$.
\end{abstract}

\maketitle

\section{Introduction}
%\subsection{}%overview
Fix a field $k$ and consider elliptic curves $E$ over $K=k(t)$.  When
$k$ is a finite field, we showed in \cite{Ulmer07b} that there are
often finite extensions $L$ of $K$ which are themselves rational
function fields (i.e., $L\cong k(u)$) such that the rank of $E(L)$ is
as large as desired.  Indeed, under a mild parity hypothesis on the
conductor of $E$ (which should hold roughly speaking in one half of
all cases), we may take extensions of the form $L=k(t^{1/d})$ with $d$
varying through integers prime to $p$.  More generally, for any
elliptic curve $E$ over $K$ with $j(E)\not\in k$ there is a finite
extension of $K$ of the form $k'(u)$ with $k'$ a finite extension of
$k$ such that $E$ obtains unbounded rank in the layers of the tower
$k'(u^{1/d})$.  Our aim in this note is to show that the situation is
completely different when $k$ is the field of complex numbers.

From now on we take $k=\C$.  If $E$ is an elliptic curve over $\C(t)$,
the {\it height\/} of $E$ is the smallest non-negative integer $d$
such that $E$ has a Weierstrass equation
$$y^2=x^3+a(t)x+b(t)$$
where $a(t)$ and $b(t)$ are polynomials of degree $\le4d$ and $\le6d$
respectively.  Our results concern elliptic curves of height $d\ge3$.

\begin{thm}\label{thm:main}
  A very general elliptic curve $E$ over $\C(t)$ of height $d\ge3$ has
  the following property: For every finite rational extension
  $L\cong\C(u)$ of $\ \C(t)$, the Mordell-Weil group $E(L)=0$.
\end{thm}

Here ``very general'' means that in the relevant moduli space, the
statement holds on the complement of a countable union of proper
closed subsets.  See Subsection~\ref{ss:very general} below for more
details.

The theorem shows in particular that there is no hope of producing
elliptic curves of large rank over $\C(u)$ by starting with a general
curve over $\C(t)$ and iteratively making rational field extensions.
What can be done with special elliptic curves remains a very
interesting open question about which we make a few speculations in
the last section.

Now we connect the theorem with the title of the paper.  Attached to
$E/\C(t)$ is a unique elliptic surface $\pi:\EE\to\P^1$ with the
properties that $\EE$ is smooth over $\C$, and $\pi$ is proper and
relatively minimal with generic fiber isomorphic to $E/\C(t)$.
Conversely, given a relatively minimal elliptic surface
$\pi:\EE\to\P^1$, its generic fiber is an elliptic curve over $\C(t)$.
The height of $E$ is then equal to the degree of the
invertible sheaf $\omega=(R^1\pi_*\O_\EE)^{-1}$ on $\P^1$.  It is well
known that the height $d$ of $E$ is 0 if and only if $\EE$ is a
product $E_0\times\P^1$; $d=1$ if and only if $\EE$ is a rational
surface; $d=2$ if and only if $\EE$ is a K3 surface; and $d\ge3$ if
and only if the Kodaira dimension of $\EE$ is 1.  See
\cite[Lecture~3]{Ulmer11} for details on the assertions in this
paragraph.

The Mordell-Weil group $E(\C(t))$ is canonically isomorphic to the
group of sections of $\pi:\EE\to\P^1$ via a map which sends a section
to its image in the generic fiber.  Suppose $L=\C(u)$ is a finite
extension of $\C(t)$ over which $E$ has a rational point.  Then we have
a diagram
\begin{equation*}
\xymatrix{\EE'\ar[r]^{\tilde f}\ar[d]^{\pi'}&\EE\ar[d]^{\pi}\\
\P^1_u\ar[r]_{f}\ar@/^1pc/[u]^s&\P^1_t}
\end{equation*}
Here $\pi':\EE'\to\P^1_u$ is the elliptic surface attached to $E/L$; it
is birational to (but not in general isomorphic to) the fiber product
$\EE\times_{\P^1_t}\P^1_u$.  The section $s$ corresponds to the
hypothesized rational point in $E(L)$.  It is clear that $\tilde
f(s(\P^1_u))$ is a rational curve on $\EE$ (i.e., a reduced and
irreducible subscheme of dimension 1 whose normalization is $\P^1$)
which is not contained in any fiber of $\pi$.  Conversely, if
$C\subset\EE$ is a rational curve not contained in any fiber of $\pi$
with normalization $g:\P^1\to C$, and if $L=\C(u)$ is the extension of
$\C(t)$ corresponding to $f=\pi\circ g:\P^1\to\P^1$, then we have the
following diagram:
\begin{equation*}
\xymatrix{\EE'\ar[r]\ar[d]^{\pi'}&\P^1_u\times\EE\ar[r]\ar[d]&\EE\ar[d]^{\pi}\\
\P^1_u\ar@{=}[r]\ar@/^1pc/[u]^s&\P^1_u\ar[r]_{f}\ar@/^1pc/[u]\ar[ur]^g&\P^1_t.}
\end{equation*}
Here the section in the middle comes from the universal property of
the fiber product, $\EE'$ is a relatively minimal desingularization of
$\P^1_u\times\EE$, and the section on the left is the unique lift of the
section in the middle.  Since a rational curve in $\EE$ not contained
in a fiber of $\pi$ maps finite-to-one to the base $\P^1_t$, we call it
a ``rational multisection.''  This discussion shows that the the
theorem above is equivalent to the following.

\begin{thm}\label{thm:main-surface}
  The only rational curves on a very general elliptic surface
  $\EE\to\P^1$ of height $d\ge3$ \textup{(}i.e., geometric genus
  $p_g=d-1\ge2$\textup{)} are the zero section and components of
  singular fibers.  Equivalently, $\EE$ has no rational multisections
  other than the zero section.
\end{thm}

The theorem is obviously false for $d\le1$ since every rational or
ruled surface contains infinitely many rational curves. The same is
expected to be true in the case $d=2$ (corresponding to K3 surfaces).
See \cite{LiLiedtke12} for the most recent results in this direction.
(Thanks to Remke Kloosterman for reminding me to mention this.)  The
case $d=1$ has been used to construct elliptic curves of relatively
high rank over $\C(t)$.  See \cite{Shioda86} and \cite{Stiller87}.

Note that the theorem is strictly stronger than the theorem of
\cite{Cox90} saying that a very general $\EE$ as above has no sections
other than the zero section.  There are a number of results in the
literature along the lines of our result, saying for example that a
general hypersurface in $\P^n$ of sufficiently large degree contains
no rational curves.  Our result is related, but necessarily trickier
because the surfaces in question certainly do contain rational curves,
namely the zero section and the components of bad fibers.
Nevertheless, a variant of an idea of Voisin \cite{Voisin96} will
allow us to show that generically there are no others.

Fedor~Bogomolov points out that a very simple construction contained
in a paper with Yuri Tschinkel \cite[Rmk.~6.11]{BogomolovTschinkel07}
yields elliptic surfaces of every height $d>2$ with no rational
multisections.  More generally, for every $g\ge0$ and all sufficiently
large $d$, they produce an elliptic fibration of degree $d$ without
multiple fibers and with no multisection of geometric genus $\le g$.
Their construction is easily modified to produce elliptic fibrations
with section but without other multisections of low genus.  Our method
is completely different, and we hope it and the conjecture it suggests
cast enough light to be of independent interest.

In the following section, we give a heuristic sketch of why the
theorem should be true which is very suggestive.  The following three
sections contain a proof of the theorem which is logically independent
of the heuristic sketch.  The final section of the paper presents
further speculations based on the heuristics in Section~2.

Much of the work on this paper was carried out during a sabbatical
year in France.  It is a pleasure to thank several institutions for
their financial support and hospitality (IHES, the Universities of
Paris VI, VII, and XI), and to thank Laurent Clozel and Marc Hindry
for their help and encouragement.  Special thanks are due to Claire
Voisin for suggesting the key idea of considering the distribution in
Section~\ref{s:distribution}.

\section{Heuristics}\label{s:heuristics}
It is known (\cite{Cox90}) that a very general elliptic surface
$\pi:\EE\to\P^1$ of height $d\ge3$ has no sections other than the zero
section, and its N\'eron-Severi group has rank 2, generated by the
zero section and the class of a fiber of $\pi$: $NS(\EE)=\langle
O,F\rangle$.  (This is also implied by Theorem~\ref{thm:main-surface}.
Our proof, although related, is independent of that of \cite{Cox90}.)
Under the intersection pairing on $NS(\EE)$, we have $O^2=-d$,
$O.F=1$, $F^2=0$.  The canonical class $K=K_\EE$ is $(d-2)F$.

Suppose $C\subseteq\EE$ is a rational curve which is not the zero
section and which is not supported in a fiber of $\pi$.  In $NS(\EE)$,
we have $C\sim eO+fF$ for certain integers $e$ and $f$, namely $e=C.F$
and $f=C.O+de$.  Since $C$ is not the zero section, $C.O\ge0$ and so
$f\ge de$.  Also, since $C$ is not the zero section, it is not a
section and so $e>1$.

Now let $L$ be the invertible sheaf $\O_\EE(eO+fF)$ with integers
$e>1$ and $f\ge de$.  It is easy to see that 
\numberwithin{equation}{section}
\begin{equation}\label{eq:pi_*L}
\pi_*L\cong\OO_{\P^1}(f)\oplus\OO_{\P^1}(f-2d)\oplus\OO_{\P^1}(f-3d)
\oplus\cdots\oplus\OO_{\P^1}(f-de)
\end{equation}
\numberwithin{equation}{subsection}%
and $R^1\pi_*L=0$.  Since we assumed that $f\ge de$, we find that 
$H^1(\EE,L)=H^2(\EE,L)=0$, and so, by Riemann-Roch,
$$h^0(\EE,L)=\chi(\EE,L)=\frac12L.(L-K)+\chi(\EE,\O).$$
(Alternatively, it is easy to compute $h^0(\EE,L)$ directly from
Equation~\eqref{eq:pi_*L}.)

This shows that the linear series $|L|=\P H^0(\EE,L)$ has dimension
$(1/2)L.(L-K)+\chi(\EE,\OO)-1$.  It is also clear that this linear
series is base-point free.

The characteristic zero Bertini theorem \cite[Cor.~10.9]{HartshorneAG}
implies that the general member of the linear series $|L|$ is a smooth
curve of genus $p_a=(1/2)L.(L+K)+1$.

Generically, one expects that the singular members of the linear
series $|L|$ are curves with at worst nodes as singularities and that
the locus of curves with $\delta$ nodes has codimension $\delta$.
Since 
$$p_a-\dim|L|=L.K-\chi(\EE,\OO)+2=(d-2)(e-1)>0,$$ 
generically we do not expect to find any rational curves in $|L|$, and
therefore generically $\EE$ should have no rational multisections
other than the zero section.

More speculatively, one might expect that the codimension in moduli of
the locus of elliptic surfaces $\EE$ of height $d\ge3$ which do admit
a rational multisection $C$ with $C.F=e>1$ might be $(d-2)(e-1)$.  But
one reason for caution here is that the locus of $\EE$ with
$NS(\EE)=\langle O,F\rangle$ is not Zariski open.  In fact it is the
complement of a divisor and a countable union of closed subvarieties
of codimension $d-1$ which is dense in the classical topology.  Also,
when the N\'eron-Severi group is more complicated, rational curves may
appear in classes other than $eO+fF$ and $\EE$ may even have smooth
rational multisections.  
%[Check that this really happens in char 0.]
%[maybe something about how/why this fails so badly in char $p$?]

\section{Set up and strategy}
In this section, we establish notation and explain the strategy to
prove Theorem~\ref{thm:main-surface}.  We consider the moduli space of
elliptic surfaces of degree $d$ and the Hilbert scheme of rational
curves on these surfaces.  Arguing by contradiction, we suppose that a
general such surface carries a rational curve and deduce 
geometric consequences of this.  In the following two sections, we
construct a distribution tangent to the rational curves which, roughly
speaking, implies that they all come from one universal rational curve.
We then deduce a contradiction from this exotic situation.

\subsection{}\label{ss:moduli}
We review the construction of the moduli space of elliptic fibrations
over $\P^1$, as in \cite{Miranda81}.

Consider the affine space $\A^{10d+2}$ of pairs of homogeneous
polynomials $(A,B)$ in $\C[t_0,t_1]$ of degrees $4d$ and $6d$
respectively.  Let $T_d$ be the open subset consisting of pairs
satisfying the conditions: (i) $4A^3+27B^2\neq0$; and (ii) for every
place $v$ of $P^1_\C$, $\ord_{v}(A)<4$ or $\ord_{v}(B)<6$.  Associated
to the pair $(A,B)$ we have $a=A(t,1)$ and $b=B(t,1)$ and the elliptic
curve over $\C(t)$ with affine Weierstrass equation
\begin{equation}\label{eq:W}
y^2=x^3+ax+b.
\end{equation}
The condition (i) guarantees that this is a smooth curve, and
condition (ii) guarantees that it has height $d$ in the sense of
Section~1.

Let $\EE$ be the minimal elliptic surface associated to $E$.  More
explicitly, let $U$ be the closed subset of $\A^1\times\P^2$ defined
by the vanishing of $x^3+axz^2+bz^3-y^2z$, let $a'=A(1,t')$,
$b'=B(1,t')$, and let $U'$ be the closed subset of $\A^1\times\P^2$
defined by the vanishing of
$x^{\prime3}+a'x'z^{\prime2}+b'z^{\prime3}-y^{\prime2}z'$.  Let $\WW$ be
the result of gluing $U\setminus\{t=0\}$ to $U'\setminus\{t'=0\}$ via
the map
$$(t',[x',y',z'])=(t^{-1},[t^{-2d}x,t^{-3d}y,z]).$$
The surface $\WW$ has an obvious projection to $\P^1$ which is proper
and relatively minimal with generic fiber $E$, but $\WW$ may have
singularities.  Indeed, it has a rational double point in each fiber
where $\Delta=-16(4A^3+27B^2)$ has a zero of order $>1$.  Resolving
these double points yields a smooth, relatively minimal elliptic
surface $\pi:\EE\to\P^1$.

There is an action of $\C^*\times\PGL_2(\C)$ on $T_d$:
$\lambda\in\C^*$ acts by $\lambda(A,B)=(\lambda^4A,\lambda^6B)$, and
$\PGL_2(\C)$ acts through its standard action on $\C[t_0,t_1]$.  Two
pairs $(A,B)$ give isomorphic elliptic surfaces $\EE$ if and only if
they lie in the same orbit of $\C^*\times\PGL_2(\C)$ acting on $T_d$.
For $d\ge2$, the orbit space $T_d/(\C^*\times\PGL_2(\C))$ has the
structure of a quasi-projective variety and is the coarse moduli space
of elliptic surfaces over $\P^1$ with a section
\cite[Cor.~5.5]{Miranda81}.  We denote it by $\MM_d$; it is
irreducible of dimension $10d-2$.

It will be convenient to introduce the Zariski open subset
$T^o_d\subseteq T_d$ consisting of pairs $(A,B)$ where
$$\Delta=-16(4A^3+27B^2)$$ 
and 
$$\Gamma=(\partial A/\partial t_0)(\partial B/\partial t_1)
   -(\partial A/\partial t_1)(\partial B/\partial t_0)$$
have distinct, simple zeros and are not divisible by $t_1$.  (That
this is indeed a non-empty Zariski open set follows from
\cite[Lemma~3.1]{CoxDonagi86}.)  The image
of $T^o_d$ in $\MM_d$ contains a Zariski open which will be denoted
$\MM_d^o$. 

\subsection{}%family
The construction above of $\EE\to\P^1$ globalizes easily:  Over
$T_d$ we have a natural family of elliptic surfaces 
\begin{equation*}
\xymatrix{\SS\ar[d]_\Pi\ar@/^2pc/[dd]^\Phi\\
T_d\times\P^1\ar[d]\\ 
T_d}
\end{equation*}
where the fiber of $\Phi$ over $(A,B)$ is the elliptic surface
$\pi:\EE\to\P^1$ associated to the elliptic curve \ref{eq:W}.  Let
$\SS^o=\Phi^{-1}(T_d^o)$ and note that on $T_d^o$ the discriminant
$\Delta$ has simple zeroes, so the fibers of $\Phi$ are just the
surfaces $\WW$ written down above (i.e., these are already smooth
surfaces and no further blowing up is needed to arrive at a smooth
model).

\subsection{}%Hilbert
Now consider the Hilbert scheme $\Hom_{T_d}(\P^1_{T_d},\SS)$ of
$T_d$-morphisms from the projective line over $T_d$ to $\SS$.  We
refer to \cite{KollarRCoAV} for the construction and general
background.  Let $\Hom^{bir}_{T_d}(\P^1_{T_d},\SS)$ be the open
subscheme representing morphisms which are birational onto their image
(i.e,, generically of degree 1), and let $MS_d$ be the open subscheme
of $\Hom^{bir}_{T_d}(\P^1_{T_d},\SS)$ representing morphisms whose
image does not lie in a fiber of $\Pi$ and is not the zero section.
($MS$ stands for multisection, although ``parameterized multisection
not lying in the zero section" would be more accurate.)

We have a diagram
\begin{equation*}
\xymatrix{MS\times\P^1_{T_d}\ar[r]\ar[d]&\SS\ar[d]^\Phi\\
MS\ar[r]&T_d\ar[d]\\
&\MM_d}
\end{equation*}
where the left and bottom horizontal arrows are the obvious
projections, the top horizontal arrow is the universal parameterized
multisection, and the other arrows were constructed above.

\subsection{}\label{ss:very general}
Theorem~\ref{thm:main-surface} is equivalent to the assertion that no
component of $MS$ dominates $\MM_d$.  Indeed, $MS$ is a countable
union of quasi-projective varieties and so has countably many
components.  The image of each component in $\MM_d$ is a constructible
set and if no component dominates, then there is a countable union of
closed subvarieties of $\MM_d$ which contains the image of
$MS\to\MM_d$.  The surfaces $\EE$ whose moduli point lies outside this
union are then the ``very general'' surfaces of the theorem.

We will eventually prove Theorem~\ref{thm:main-surface} by
contradiction.  Thus, assume that some component of $MS$ dominates
$\MM_d$.  Taking a slice of $MS$ and passing to Zariski opens in $MS$
and $\MM_d$, we arrive at a diagram
\begin{equation*}
\xymatrix{\BB\times\P^1\ar[r]\ar[d]&\SS^o\ar[d]^\Phi\\
\BB\ar[r]&T^o_d\ar[d]\\
&\MM^o_d}
\end{equation*}
where $\BB$ is irreducible and $\BB\to\MM_d^o$ is \'etale, dominant, and
lies over the smooth locus of $\MM_d^o$.

To simplify notation for the rest of the proof, we write $\XX$ for the
fiber product $\BB\times_{T_d^o}\SS^o$ so that we have a diagram
\begin{equation}\label{eq:bad-family}
\xymatrix{
\BB\times\P^1=\CC\ar[rr]^i\ar[rd]_{pr_\BB}&&\XX\ar[ld]^\Phi\\
&\BB.&}
\end{equation}

The following proposition sums up the discussion in this section.
\begin{prop}\label{prop:family}
  If Theorem~\ref{thm:main-surface} is false, then there exists a
  smooth, irreducible, quasi-projective variety $\BB$ of dimension
  $10d-2$ which is \'etale over a closed subset of $T_d^o$ and a
  diagram \eqref{eq:bad-family} such that the fiber of $\Phi$ over
  each $b\in \BB$ is a Weierstrass fibration $\pi_b:\EE_b\to\P^1$ of
  height $d$, the induced morphism $\BB\to\MM_d$ is \'etale and
  dominant, and the map $\{b\}\times\P^1\to\EE_b$ is a multisection of
  $\pi_b$.
\end{prop}

In the next section we will collect some differential information
which severely restricts the image of the upper horizontal map $i$.
In the following section, we will see that the restrictions are so
severe that we arrive at a contradiction.

\section{Constructing a distribution}\label{s:distribution}

%\subsection{}
Throughout this section, we assume that Theorem~\ref{thm:main-surface} is
false and we let $\XX\to B$ be the family of elliptic surfaces
constructed in Proposition~\ref{prop:family}.
 
Let $N=\dim \BB=10d-2$.  In this section we will first show that
$\Phi_*(\Omega^{N+1}_\XX)$ is locally free on $\BB$ of rank $(d-2)N$.
Then we will show that for all $x$ in a certain open subset of
$\XX$, the sections of $\Phi_*(\Omega^{N+1}_\XX)$ generate a
codimension 2 subspace of the fiber of $\Omega^{N+1}_\XX$ at $x$.
This subspace determines a codimension 2 subspace of the tangent space
$T_{\XX,x}$, i.e., we have a distribution of rank $N$ over an open
subset of $\XX$.

If $\omega$ is a section of $\Phi_*(\Omega^{N+1}_\XX)$ over some open
in $B$, then $i^*\omega$ is a section of
$$pr_{\BB*}(\Omega^{N+1}_\CC)\cong K_\BB\tensor
pr_{\BB*}(\Omega^1_{\P^1})=0.$$ The equality $i^*\omega=0$ amounts to
saying that the image of $i$ is tangent to the distribution above
(i.e., $di(T_{\CC,c})$ contains the subspace).
This imposes strong differential conditions on the image of
$i:\CC\to\XX$ and will eventually lead to a contradiction.

\subsection{}
Start with the exact sequence
$$0\to T_{\XX/\BB}\to T_{\XX}\to\Phi^*(T_\BB)\to0$$
of locally free sheaves on $\XX$, tensor with
$K_\XX\cong\Phi^*(K_\BB)\tensor\Omega^2_{\XX/\BB}$, and take the direct
image under $\Phi$.  Noting that $\Phi_*(\Omega^1_{\XX/\BB})=0$, we
obtain an exact sequence
$$0\to\Phi_*(\Omega^{N+1}_\XX)\to K_\BB\tensor\Phi_*(\Omega^2_{\XX/\BB})
\tensor T_\BB\to K_\BB\tensor R^1\Phi_*(\Omega^1_{\XX/\BB})$$ of sheaves on
$\BB$.  The middle and right sheaves are locally free of ranks
$(d-1)N$ and $10d$ respectively, and the map between them is
the identity on $K_\BB$ tensored with the Kodaira-Spencer map
$\Phi_*(\Omega^2_{\XX/\BB})\tensor T_\BB\to
R^1\Phi_*(\Omega^1_{\XX/\BB})$.

For each $b\in \BB$, the classes of the zero section and a fiber of
$\pi_b$ are linearly independent in $H^1(\EE_b,\Omega^1_{\EE_b/\C})$
and we get two everywhere independent sections of
$R^1\Phi_*(\Omega^1_{\XX/\BB})$ over $\BB$.  We write
$R^1\Phi_*(\Omega^1_{\XX/\BB})_0$ for their orthogonal complement.
This is a locally free subsheaf of $R^1\Phi_*(\Omega^1_{\XX/\BB})$ of
rank $N$.  Since the classes of the zero section and fiber live over
the whole base $\BB$, the image of the Kodaira-Spencer map lands in 
$R^1\Phi_*(\Omega^1_{\XX/\BB})_0$.  Thus we have an exact sequence
$$0\to\Phi_*(\Omega^{N+1}_\XX)\to K_\BB\tensor\Phi_*(\Omega^2_{\XX/\BB})
\tensor T_\BB\to K_\BB\tensor (R^1\Phi_*(\Omega^1_{\XX/\BB}))_0.$$

Next, we will show that the right hand map is
surjective and so the kernel is locally free.

\begin{prop}\label{prop:locally-free}
  The sheaf $\Phi_*(\Omega^{N+1}_\XX)$ on $B$ is locally free of rank
  $(d-2)N$. 
\end{prop}

\begin{proof}
Consider the graded rings
$$S_0=\C[t_0,t_1]\subseteq S=\C[t_0,t_1,x,y]$$ 
where the degrees of $t_0$ and $t_1$ are 1, the degree of $x$ is $2d$,
and the degree of $y$ is $3d$.  Let $\P=\P(1,1,2d,3d)$ be the
weighted projective space $\Proj(S)$.  (See \cite{Dolgachev82} for
general background on weighted projective varieties.)

Fix a point $b\in \BB$ mapping to a pair $(A,B)\in T^o_d$ and let
$\EE=\Phi^{-1}(b)$, the elliptic surface associated to $(A,B)$.  

Let $F=x^3+Ax+B-y^2$, a homogeneous element of $S$ of degree $6d$.  It
is easy to see that the weighted hypersurface $\Proj(S/(F))$ is
isomorphic to $\EE$ with its zero section collapsed to a point.  (Here
we use that $(A,B)$ is in $T^o_d$ so that $\Delta$ has distinct roots
and $\EE=\WW$ in the notation of Subsection~\ref{ss:moduli}.)  This
construction globalizes (i.e., we have a family over $T_d^o$), so the
deformations of $\EE$ preserving the zero section are the same as
those of the hypersurface.

To make this more precise, we use subscripts to denote partial
derivatives, so, e.g., $F_{t_0}:=\partial F/\partial t_0$.  Let $R$ be
the Jacobian ring $S/(F_{t_0},F_{t_1},F_x,F_y)$ and use superscripts
to denote the graded pieces.  Then works of Carlson-Griffiths and
Saito explained in \cite{CoxDonagi86} show that $H^1(\EE,T_\EE)_0$,
the subspace of $H^1(\EE,T_\EE)$ corresponding to first order
deformations which preserve the classes of the zero section and the
fiber, is isomorphic to $R^{6d}$.  Explicitly, an element of $R^{6d}$
can be written in the form $\tilde Ax+\tilde B$ where $\tilde A\in
S_0^{4d}$ and $\tilde B\in S_0^{6d}$.  The corresponding first order
deformation of $\Proj(S/(F))$ is the subscheme of
$\P\times\Spec\C[s]/(s^2)$ defined by the vanishing of
$$x^3+(A+s\tilde A)x+(B+s\tilde B)-y^2.$$

It is also explained in \cite{CoxDonagi86} that
$H^0(\EE,\Omega^2_\EE)\cong R^{d-2}$.  Explicitly, $p(t_0,t_1)\in
S_0^{d-2}=R^{d-2}$ corresponds to the 2-form on $\EE$ which on the
affine piece $\Spec\C[t,x,y]/(x^3+ax+b-y^2)$ is equal to
$p(t,1)dt\,dx/2y$.

One also knows by \cite{CoxDonagi86} that
$H^1(\EE,\Omega^1_\EE)_0\cong R^{7d-2}$ and that the Kodaira-Spencer
map
$$H^0(\EE,\Omega^2_\EE)\tensor 
H^1(\EE,T_\EE)_0\to H^1(\EE,\Omega^1_\EE)_0$$
can be identified with the map 
$$R^{d-2}\tensor R^{6d}\to R^{7d-2}$$
induced by multiplication in $S$.  Thus to compute the fiber at $b$ of
the sheaf $\Phi_*(\Omega^{N+1}_\XX)$, we should compute the kernel of
the multiplication map $R^{d-2}\tensor R^{6d}\to R^{7d-2}$.

We will make this more explicit following Cox and Donagi.  Recall that
since $(A,B)\in T_d^o$, we have that $t_1$ does not divide $\Gamma$.
Lemma 2.4 of \cite{CoxDonagi86} says that for $\lambda\in\C$ and
$6d-2\le j\le 8d-1$, the kernel of the multiplication map
$(t_0-\lambda t_1):R^{j}\to R^{j+1}$ is non-zero if and only if
$t_0-\lambda t_1$ divides $\Gamma=A_{t_0}B_{t_1}-A_{t_1}B_{t_0}$.  Let
$\lambda_i$, $i=1,\dots,10d-2$, be the roots of $\Gamma(t,1)$ and
choose a non-zero vector $v_i\in R^{6d}$ in the kernel of $t_0-\lambda
t_1$.  Since the dimension of $R^{6d}$ is $10d-2$, simple linear
algebra shows that $\{v_i\}$ is a basis of $R^{6d}$.  Let
$w_i=t_1^{d-2}v_i\in R^{7d-2}$.  Since $t_1$ does not divide $\Gamma$,
it is clear that $\{w_i\}$ is a basis of $R^{7d-2}$.  In terms of
these bases, it is evident that the multiplication map
$$R^{d-2}\tensor R^{6d}\to R^{7d-2}$$
sends $p(t_0,t_1)\tensor v_i$ to $p(\lambda_i,1)w_i$.  Therefore,
$$\ker\left(R^{d-2}\tensor R^{6d}\to R^{7d-2}\right)\cong
\oplus_{i=1}^{N}W_i\tensor \C v_i$$ where $W_i=\{p\in
R^{d-2}|p(\lambda_i,1)=0\}$, a codimension 1 subspace of $R^{d-2}$.
(It is an interesting question whether the directions $v_i$ in moduli
have any geometric significance.)

This calculation shows that for every $b\in \BB$, the 
map $R^{d-2}\tensor R^{6d}\to R^{7d-2}$ is surjective and that the
dimension of the fiber of $\Phi_*(\Omega^{N+1}_\XX)$ at $b$ is $(d-2)N$.
Thus the fibers of the coherent the sheaf $\Phi_*(\Omega^N_{\XX/\BB})$
have constant dimension, and it is therefore locally free of rank
$(d-2)N$.  
\end{proof}

\subsection{}
Our next task is to write down explicitly sections of
$\Phi_*(\Omega^{N+1}_\XX)$ in a neighborhood of any $b\in
B$.  

Fix a point $b\in \BB$ mapping to $(A,B)\in T^o_d$.  We
constructed above a basis $v_1,\dots,v_N$ of the tangent space
$T_{B,b}$ indexed by the roots $\lambda_1,\dots,\lambda_N$ of
$\Gamma$.  Let $s_1,\dots,s_N$ be a system of parameters at $b$ such
that $\partial/\partial s_i|_b=v_i$.  

For notational simplicity, we choose a root $\lambda_i$ and consider
the 1-parameter deformation of $\EE_b$ in the corresponding direction
$v_i$.  Thus, let $\lambda=\lambda_i$ and $s=s_i$.  Since $\lambda$ is
a root of $\Gamma$, there is a non-trivial solution $(\alpha,\beta)$
of
$$
\begin{pmatrix}
  A_{t_0}(\lambda,1)&  A_{t_1}(\lambda,1)\\
  B_{t_0}(\lambda,1)&  B_{t_1}(\lambda,1)
\end{pmatrix}
\begin{pmatrix}
  \alpha\\ \beta
\end{pmatrix}
=0.$$
Fix a solution and define $\tilde Ax+\tilde B\in R^{6d}$ by
$$\tilde Ax+\tilde B=t_1^2\frac{\alpha F_{t_0}(\lambda,1)+\beta
  F_{t_1}(\lambda,1)}{t_0-\lambda t_1}$$
where, as usual, $F=x^3+Ax+B-y^2$.

The 1-parameter deformation of $\EE_b$ in the direction $v=v_i$
corresponds to the family of hypersurfaces $\YY$ defined by
$$y^2=x^3+(A+s\tilde A)x+(B+s\tilde B)$$
in $\A^1_s\times\P$.  Let $\phi:\ZZ\to\A^1_s$ be the corresponding
family of elliptic surfaces with section, so that
$\phi^{-1}(0)=\EE_b$.  

It will also be convenient to consider the affine open
subset $\ZZ^o\subset\ZZ$ defined by
$$y^2=x^3+(a+s\tilde a)x+(b+s\tilde b)$$
in the $\A^4$ with coordinates $x,y,s,t$ where $t=t_0/t_1$,
$a(t)=A(t,1)$, etc.  

\begin{prop}\label{prop:sections}
  Let $p(t)\in\C[t]$ be a polynomial of degree $\le d-3$ and let
$$\omega=p(t)
\left(\frac{(t-\lambda)dtdx}{2y}+\frac{(\alpha-\beta t)dsdx}{2y}
  +\frac{(2d\,\beta x)dsdt}{2y}\right).$$
Then $\omega$ is a regular 2-form on a neighborhood of $s=0$ in
$\ZZ^o$ and it extends to a section of $\phi_*(\Omega^2_\ZZ)$ over a
neighborhood of $0\in\A^1_s$.  Similarly,
$$ds_1\cdots\widehat{ds_i}\cdots ds_N\wedge\omega$$
extends to a section of $\Phi_*(\Omega^{N+1}_\XX)$ in a neighborhood
of $b\in B$.
\end{prop}

\begin{proof}
  Let $f=x^3+(a+s\tilde a)x+(b+s\tilde b)-y^2$.  It is obvious that
  $f_y\omega=-2y\omega$ is everywhere regular on $\ZZ^o$.
  Straightforward calculations (sketched below) show that $f_x\omega$,
  $f_s\omega$, and $f_t\omega$ are also everywhere regular on $\ZZ^o$.
  Since $\ZZ^o$ is smooth in a neighborhood of $s=0$, this shows that
  $\omega$ is a regular 2-form on that neighborhood.  That it extends
  to a neighborhood of $s=0$ in $\ZZ$ follows from similar
  calculations.  The last sentence of the proposition follows
  immediately from what precedes it.

It remains to sketch the calculation showing that $f_x\omega$, etc.,
are regular.   Noting that
$$0=df=f_xdx+f_ydy+f_sds+f_tdt$$
on $\ZZ^o$, it follows that the 1-form
$$dy=\frac{f_xdx+f_sds+f_tdt}{2y}$$
is everywhere regular on $\ZZ^o$.  From this we deduce that the
following 2-forms are regular on $\ZZ^o$:
$$\frac{f_tdtdx+f_sdsdx}{2y},\quad
\frac{f_xdxdt+f_sdsdt}{2y},\quad\text{and}\quad
\frac{f_xdxds+f_tdtds}{2y}.$$
The last ingredient in the equality defining $\tilde a$ and $\tilde
b$:
\begin{equation}\label{eq:lambda-tilde}
(t-\lambda)f_s=(\alpha-\beta t)f_t+(2d)\beta(3y^2-xf_x).
\end{equation}
Using this equality, we may rewrite $f_x\omega$, etc., as combinations
of the everywhere regular 2-forms above.  This completes the sketch of
the calculation and the proof of the proposition.
\end{proof}

\subsection{}
Note that we have constructed $(d-2)N$ linearly independent sections
of $\Phi_*(\Omega^{N+1}_\XX)$ in a neighborhood of $b$.  The dimension
count in Proposition~\ref{prop:locally-free} shows that they are a
basis.  It is evident that at a general point $x\in\XX$, these
sections span a subspace of the fiber of $\Omega^{N+1}_\XX$ of
dimension $N$.  

For reasons which will be clear just below, we introduce
$U\subset\XX$, defined to be the open subset of $\XX$ where $\Gamma$
does not vanish.  In slightly different notation, this is the open set
where the coefficient $(t-\lambda)$ in $\omega$ does not vanish.

Suppose $x\in U$.  Choosing an identification
$\bigwedge^{N+2}T^*_{\XX,x}\cong\C$ yields an identification of the
  $\bigwedge^{N+1}T^*_{\XX,x}$, the fiber of $\Omega^{N+1}_\XX$ at
    $x$, with $T_{\XX,x}$.  The subspace constructed above corresponds
    to a subspace of $T_{\XX,x}$ of dimension $N$ which is independent
    of the choice.  We write $\DD$ for the resulting distribution of
    rank $N$ on $U\subset \XX$.

The existence of this distribution $\DD$ of codimension 2 is the analog of
the global generation result of \cite{Voisin96}, taking into account
the two classes of rational curves lying on every $\EE_b$ in our family.

Note that since $t-\lambda\neq0$ on $U$, we
have that $\DD_x\subset T_{\XX,x}$ projects isomorphically onto the
tangent space of $B$.  Cox and Donagi showed
\cite[Prop.~3.3]{CoxDonagi86} that the zeros $t-\lambda$ of $\Gamma$
are exactly the branch points away from $j=0$ and $j=1728$ of the
$j$-invariant mapping $j:\P^1_t\to\P^1_j$ induced by the fibration
$\EE_b\to\P^1_t$.  Since the image of $i$ is a multisection, it
meets $U$.

\subsection{}
To finish this section, we simply note that the condition
$i^*(\omega)=0$ observed at the beginning of the section is equivalent
to the assertion that $i(\CC)$ is tangent to $\DD$.  More precisely,
for all $c$ in some non-empty open subset of $\CC$, we have $i(c)\in
U$ and 
$$di(T_{\CC,c})\supset \DD_{i(c)}.$$

The following summarizes the discussion of this section.

\begin{prop}
  Let $U\subset\XX$ be the open subset were $\Gamma$ does not vanish.
  The construction above yields a distribution $\DD$ on $U$ of
  dimension $N$ such that for all $x\in U$, $\DD_x\subset T_{\XX,x}$
  projects isomorphically onto $T_{\BB,\Phi(x)}$.  The image of
  $i:\CC=\BB\times\P^1\to\XX$ meets $U$, and for all $c\in\CC$ such
  that $i(c)\in U$, we have $di(T_{\CC,c})\supset \DD_{i(c)}.$
\end{prop}

The corollary will turn out to be such a strong restriction on $i$
that it cannot exist.

\section{Deducing a contradiction}
We continue to assume that Theorem~\ref{thm:main-surface} is false,
and we consider the diagram
$$\xymatrix{
  \BB\times\P^1=\CC\ar[rr]^i\ar[rd]_{pr_\BB}&&\XX\ar[ld]^\Phi\\
  &\BB&}$$
and the distribution $\DD$ on $U\subset\XX$ constructed above.

The miraculous fact is that the distribution $\DD$ has two first
integrals.  Indeed, consider the rational functions on $\XX$ given by
$$j=\frac{-2^{12}3^3A^3}{\Delta}=\frac{6912A^3}{4A^3-27B^2}\quad
\text{and}\quad k=\frac{xA}{B}.$$

\begin{prop}\label{prop:integrals}
  Generically, $\DD$ is tangent to the fibers of $j$ and $k$.  More
  precisely, for all $x\in U$, we have
$$\DD_x\subset T_{j^{-1}(j(x))}\quad\text{and}\quad\DD_x\subset T_{k^{-1}(k(x))}.$$
\end{prop}

\begin{proof}
The assertion can be checked by a straightforward calculation.  Again
for notational convenience we work with 1-parameter deformations
in the directions corresponding to roots of $\Gamma$.  Choosing such a
direction and using the notation introduced just before
Proposition~\ref{prop:sections}, it suffices to show that $\omega\wedge
dj=\omega\wedge dk=0$ on $\ZZ^o$.  One computes that
$$dj=\frac{2^{16}3^6a^2b}{\Delta^2}\left((3a'b-2ab')dt+(3\tilde
  ab-2a\tilde b)ds\right)$$
and
$$dk=\frac{adx}b+\frac{x}{b^2}\left((a'b-ab')dt+(\tilde ab-a\tilde
  b)ds\right)$$
where $a'$ and $b'$ are the derivatives of $a$ and $b$ with respect
to $t$ evaluated at $s=0$.  The desired vanishing then falls out from
a simple calculation using the equation~\eqref{eq:lambda-tilde}.
\end{proof}

Note that generically the fibers of $j$ and $k$ are transverse, so
the Proposition gives an alternative characterization of $\DD$.

\begin{remss}
  We proved the Proposition by calculating $\omega$ (and thus $\DD$)
  explicitly in Proposition~\ref{prop:sections}, then computing
  $\omega\wedge dj$ and $\omega\wedge dk$ explicitly.  
%Since the statement is so natural, i
  It would be more satisfying to have a conceptual proof.
\end{remss}

\begin{cor}\label{cor:Z}
The image of the composed rational map
$$\CC\overset{i}{\longrightarrow}
\XX\overset{j\times k}{\longrightarrow}\P^1_j\times\P^1_k$$
is a rational curve which meets $\A^1_j\times\A^1_k$ and maps with
positive degree to $\P^1_j$
\end{cor}

\begin{proof}
  Indeed, Proposition~\ref{prop:integrals} shows that the fibers of
  $(j\times k)\circ i$ have dimension at least $N$, so its image is a
  curve or a point.  It cannot be a point because for any $b\in B$,
  the image of $\{b\}\times\P^1$ in $\EE_b$ is a rational multisection
  which is not the zero section, and the image of this curve in
  $\P^1_j\times\P^1_k$ maps surjectively to the factor $\P^1_j$.  It
  is thus a curve which maps with positive degree to $\P_j^1$, and it is
  clearly rational.  Moreover, the subset of $\EE_b$ where $k=\infty$
  is the zero section, so the image of $(j\times k)\circ i$ meets the
  finite part of $\P^1_j\times\P^1_k$.
\end{proof}

\begin{proof}[Proof of Theorem~\ref{thm:main-surface}]
  We are now ready to deduce a contradiction.  Let $F$ be the $\P^1$
  bundle over $\P^1$ given by
$$F=\P\left(\OO_{\P^1}(-2d)\oplus\OO_{\P^1}\right)\to\P^1_t.$$
We write $x$ for the ``vertical'' coordinate on $F$ over $\A^1_t$, so
that $F$ is a compactification of $\A^2_{t,x}$.  For each $b\in \BB$
corresponding to $(A,B)$, the elliptic surface $\EE_b$ is the double
cover of $F$ branched along $x=\infty$ and $x^3+Ax+B=0$.

All this fits together into the diagram
$$\xymatrix{\XX\ar[r]^\Upsilon\ar[d]
  &\BB\times F\ar[d]\ar@{-->}[r]^\Psi&\P^1_j\times\P^1_k\ar[d]\\
\BB\times\P^1_t\ar@{=}[r]\ar[d]&\BB\times\P^1_t\ar[r]\ar[d]&\P^1_j\\
\BB\ar@{=}[r]&\BB}$$
whose fiber over $b\in\BB$ is
$$\xymatrix{\EE_b\ar[r]^{\Upsilon_b}\ar[d]
  &F\ar[d]\ar@{-->}[r]^{\Psi_b}&\P^1_j\times\P^1_k\ar[d]\\
\P^1_t\ar@{=}[r]&\P^1_t\ar[r]&\P^1_j.}$$

Let $Z\subset\P^1_j\times\P^1_k$ be the rational curve of
Corollary~\ref{cor:Z}.  If Theorem~\ref{thm:main-surface} is false,
then for every $b\in\BB$, the inverse image
$\Upsilon_b^{-1}\left(\Psi^{-1}_b\left(Z\right)\right)$ contains a
component which is a rational curve (namely, the image of
$\{b\}\times\P^1$ under $i$).  Given that we have so much flexibility
in choosing the branching data in the diagram above, this is too much
to hope for and will lead to a contradiction.  In fact, we will show
that for generic $b$, every component of the curve
$\Upsilon_b^{-1}\left(\Psi^{-1}_b\left(Z\right)\right)$ has
positive geometric genus.

To see this, let us fix a $b\in\BB$ corresponding to
$(A,B)$ and consider the diagram
$$\xymatrix{F\ar[d]\ar@{-->}[r]^{\Psi_b}&\P^1_j\times\P^1_k\ar[d]\\
  \P^1_t\ar[r]&\P^1_j.}$$
The map $\P^1_t\to\P^1_j$ sends $t$ to
$j=6912A(t)^3/(4A(t)^3+27B(t)^2)$.  For a generic $b$ this has degree
$12d$ and is branched in triples over $j=0$, in pairs over $j=1728$,
and has a single ramification point with index $e=2$ over $10d-2$
other values of $j$.  As $b$ varies, these $10d-2$ points fill out an
open set in the $(10d-2)^{th}$ symmetric power of $\P^1$.  (See \cite[3.1
and 3.3]{CoxDonagi86}.)  The diagram is commutative, and we have
$\Psi_b^*(k)=xA(t)/B(t)$.  One sees easily that $\Psi_b$ is finite
away from the zeroes of $A$ and $B$, it is undefined at the points
$\{x=0,B=0\}$ and $\{x=\infty,A=0\}$, it collapses the other points on
$A=0$ to $(j,k)=(0,0)$, and it collapses the other points on $B=0$ to
$(j,k)=(1728,\infty)$.  
For all $b$, the branch curve
$\{x=\infty\}\cup\{x^3+Ax+B=0\}$ in $F$ maps to the curve 
$$\{k=\infty\}\cup\{(6912-4j)k^3+27j(k+1)=0\}
\subset\P^1_j\times\P^1_k.$$ Write $Y$ for the component
$\{(6912-4j)k^3+27j(k+1)=0\}$ above, and note that $Y$ is a rational
curve projecting isomorphically onto $\P^1_k$.

We now consider two cases depending on whether the degree of
$Z\to\P^1_j$ is 1 or $>1$.  If the degree is 1, then $Z$ meets the
curve $Y$ transversely at exactly one point.  Moreover,
$\Psi^{-1}_b(Z)$ is an irreducible (rational) curve, as it projects to
$\P^1_t$ with degree 1.  Also, it meets the branch locus $x^3+Ax+B$
transversely in at least $4d$, $6d$, or $12d$ points; the three cases
correspond to the intersection of $Z$ and $Y$ being at $(j,k)=(0,0)$,
$(1728,\infty)$ or elsewhere.  It follows that
$\Upsilon_b^{-1}(\Psi_b^{-1}(Z))$ is an irreducible curve of positive
geometric genus.  This is a contradiction.

Now consider the case where the degree of $Z\to\P^1_j$ is $>1$.  We
claim that for generic $b$, $\Psi^{-1}_b(Z)$ is irreducible of
positive geometric genus, and therefore
$\Upsilon_b^{-1}(\Psi^{-1}_b(Z))$ has no rational components, again a
contradiction.  To see this, we note that
$\Upsilon_b^{-1}(\Psi^{-1}_b(Z))$ is birational to the fiber product
of $Z\to\P^1_j$ and $\P^1_t\to\P^1_j$.  First, we show that this fiber
product is irreducible.  As $b$ varies, most of the branch locus of
$F\ratto\P^1_j\times\P^1_k$ varies.  More precisely, $j=0$ and
$j=1728$ are always in the branch locus of $\P^1_t\to\P^1_j$, but the
other $10d-2$ branch points vary.  It follows that for most $b$, the
branch locus of $Z\to\P^1_j$ and that of $\P^1_t\to\P^1_j$ have at
most two points in common.  If they have none in common, it is
immediate that the fiber product is irreducible.  If they have one in
common, then \cite[Prop.~3.1]{Pakovich11} shows the fiber product is
irreducible.  If they have two points in common,
\cite[Prop.~3.3]{Pakovich11} shows that the fiber product is
irreducible unless $f:Z\to\P^1_j$ and $g_b:\P^1_t\to\P^1_j$ are
related as follows: There are rational functions $f_1$ and $g_{1,b}$
and a $d>1$ such that $f=\mu\circ z^d\circ f_1$ and $g_b=\mu\circ
z^d\circ g_{1,b}$ where $\mu$ is a Mobius transformation.  It is clear
that for generic $b$, $g_b$ does not have this form, so the fiber
product is irreducible.  Finally, we let $W$ be the normalization of
the fiber product $\P^1_t\times_{\P^1_j}Z$ and use the Riemann-Hurwitz
formula for $W\to Z$ to show that the geometric genus of $W$ is
positive.  Since the degree is $12d$, we must show that the
ramification term has degree $>24d-2$.  We may choose $b$ so that
$10d-2$ of the branch points of $\P^1_t\to\P^1_j$ are distinct from
any critical values of $Z\to\P^1_j$.  If $\delta$ is the degree of
$Z\to\P^1_j$, we get a contribution of $\delta(10d-2)$ to the
ramification term.  If $\delta>2$ we are done.  If $\delta=2$, then
points of $W$ over the triple ramification points of $\P^1_t\to\P^1_j$
contribute at least $2\cdot4d$ more, so again we have positive
geometric genus.  Thus in all cases, $W$ is an irreducible curve of
positive geometric genus and we again reach a contradiction.

This completes the proof of the theorem.
\end{proof}

\begin{rem}
  The results of \cite{Pakovich11} used above also allow one to show
  that for generic $b$ and $b'$, with corresponding maps
  $\P^1_t\to\P^1_j$ and $\P^1_{t'}\to\P^1_j$, the fiber product
  $\P^1_t\times_{\P^1_j}\P^1_{t'}$ is irreducible.  On the other hand,
  given some other proof of this fact, we could deduce the
  irreducibility of $W$ above by a ``linear disjointness'' argument.
  In other words, if we know that for general $b$ and $b'$, $\C(t)$
  and $\C(t')$ are linearly disjoint extensions of $\C(j)$ (by which
  we mean that $\C(t)\tensor_{\C(j)}\C(t')$ is a field), then we may
  deduce that for general $b$, $\C(t)$ and $\C(Z)$ are linearly
  disjoint, and this is equivalent to $W$ being irreducible.
\end{rem}

\begin{rem}
  As written, the proof of Theorem~\ref{thm:main-surface} relies on
  transcendental arguments via the results of \cite{Pakovich11}.
  Nevertheless, I believe that it also holds in characteristic $p>0$.
  This would not contradict the results in characteristic $p$
  mentioned in the introduction.  Indeed, all the constructions of
  infinitely many rational curves I know of depend crucially on
  working over a \emph{finite} field, i.e., they do not generalize to
  positive dimensional families.  Thus they may all be removed by
  removing countably many proper closed subvarieties in moduli.
\end{rem}

\section{Further speculation on rational curves}
A famous conjecture of Lang in \cite{Lang86}  %5.6
asserts that there should be only finitely
many rational curves on a surface of general type.  Since these curves
do not move, it is equivalent to conjecture that their degrees are
bounded, where we measure degree with the canonical bundle (i.e.,
$\deg(C)=C.K$).  

It is certainly false that there are in general only finitely many
rational curves on an elliptic surface.  Indeed, as soon as we have
one non-torsion rational multisection, we can make infinitely many
more by using the group structure.  However, the heuristics in
Section~\ref{s:heuristics} suggest that the degree over the base of a
rational curve on an elliptic surface $\EE$ of Kodaira dimension 1
should be bounded.  This is equivalent to asserting that the integer
$C.K$ should be bounded as $C$ runs through rational curves on $\EE$.

If $\SS$ is a surface of Kodaira dimension $\le0$, then it is
immediate that $C.K\le 0$ for all rational curves on $\SS$.

Combining these observations, we make the following speculation.

\begin{conj}
  Let $\SS$ be a smooth projective surface over the complex numbers
  with canonical bundle $K$.  Then the set of integers
$$\left\{K.C\,\left|\,C\subset\SS\text{ a rational
      curve}\right.\right\}$$
is bounded above.
\end{conj}

This conjecture seems to the author to be of topological nature.
Indeed, if we restrict to smooth rational curves, then it is a
consequence of the Bogomolov-Miyaoka-Yau inequality.  On the other
hand, the results of \cite{Ulmer14a} show that it is false in
characteristic $p$ for any prime $p>0$.

Applied to the case of an elliptic surface with section and the
associated elliptic curve, the conjecture would imply the following
statement: Suppose that $E$ is an elliptic curve over $K_0=\C(t)$ with
$j(E)\not\in\C$ and suppose that $K_0\subset K_1\subset
K_2\subset\cdots$ is a sequence of finite extensions of $K_0$ such
that each $K_i$ is a rational function field.  Then the rank of
$E(K_i)$ is bounded independently of $i$.  The special case where
$K_i=\C(t^{1/i})$ is labelled a ``conjecture(?)'' in
\cite{Silverman00}.  The heuristics above suggest that it may hold in
much greater generality.

\bibliography{database}{}
\bibliographystyle{alpha}

\end{document}

%% file: formatting.tex
\numberwithin{equation}{subsection}

\swapnumbers
\theoremstyle{plain}
\newtheorem{thm}[subsection]{Theorem}

\newtheorem{prop}[subsection]{Proposition}

\newtheorem{cor}[subsection]{Corollary}

\newtheorem{conj}[subsection]{Conjecture}

\theoremstyle{definition}

\theoremstyle{remark}
\newtheorem{rem}[subsection]{Remark}
\newtheorem{remss}[subsubsection]{Remark}

%% file: macros.tex
\usepackage[OT2,T1]{fontenc}
\DeclareSymbolFont{cyrletters}{OT2}{wncyr}{m}{n}
\DeclareMathSymbol{\sha}{\mathalpha}{cyrletters}{"58}

\newcommand{\BB}{\mathcal{B}}
\newcommand{\CC}{\mathcal{C}}

\newcommand{\DD}{\mathcal{D}}

\renewcommand{\SS}{\mathcal{S}}

\newcommand{\XX}{\mathcal{X}}

\newcommand{\YY}{\mathcal{Y}}

\newcommand{\ZZ}{{\mathcal{Z}}}
\newcommand{\WW}{{\mathcal{W}}}

\renewcommand{\O}{\mathcal{O}}

\newcommand{\EE}{\mathcal{E}}

\newcommand{\MM}{\mathcal{M}}

\newcommand{\OO}{\mathcal{O}}

\newcommand{\ratto}{{\dashrightarrow}}

\newcommand{\C}{\mathbb{C}}
\newcommand{\A}{\mathbb{A}}
\renewcommand{\P}{\mathbb{P}}

\newcommand{\tensor}{\otimes}

\newcommand{\PGL}{\mathrm{PGL}}

\DeclareMathOperator{\ord}{ord}

\DeclareMathOperator{\Hom}{Hom}

\DeclareMathOperator{\Spec}{Spec}
\DeclareMathOperator{\Proj}{Proj}

\def\clap#1{\hbox to 0pt{\hss#1\hss}}